\documentclass{article}

\usepackage[english]{babel}

\usepackage[letterpaper,top=2cm,bottom=2cm,left=3cm,right=3cm,marginparwidth=1.75cm]{geometry}

\usepackage{amsmath}
\usepackage{graphicx}
\usepackage[colorlinks=true, allcolors=blue]{hyperref}
\usepackage{amsmath, amssymb, amsthm, amsfonts}
\usepackage{fullpage}
\usepackage{hyperref}
\usepackage{enumitem}
\usepackage{mathtools}
\usepackage[capitalise]{cleveref}

\theoremstyle{definition}

\theoremstyle{theorem}

\newtheorem{thm}{Theorem}[section]

\newtheorem{lemma}[thm]{Lemma}
\newtheorem{conj}[thm]{Conjecture}

\theoremstyle{remark}

\newtheorem{claim}{Claim}

\newcommand{\cC}{\mathcal{C}}

\newcommand{\cH}{\mathcal{H}}

\DeclarePairedDelimiter\norm{\lVert}{\rVert}%
\DeclarePairedDelimiter\abs{\lvert}{\rvert}%
\makeatletter
\let\oldabs\abs
\def\abs{\@ifstar{\oldabs}{\oldabs*}}
\let\oldnorm\norm
\def\norm{\@ifstar{\oldnorm}{\oldnorm*}}
\makeatother

\title{When are off-diagonal hypergraph Ramsey numbers polynomial?}
\author{David Conlon\thanks{Department of Mathematics, California Institute of Technology, Pasadena, CA 91125. Email: dconlon@caltech.edu. Research supported by NSF Awards DMS-2054452 and DMS-2348859.} \and
Jacob Fox\thanks{Department of Mathematics, Stanford University, Stanford, CA 94305. Email: jacobfox@stanford.edu. Research supported by NSF Award DMS-2154129.} \and
Benjamin Gunby\thanks{Department of Mathematics, Rutgers University, Piscataway, NJ 08854. Email: bg570@rutgers.edu.}\and
Xiaoyu He\thanks{School of Mathematics, Georgia Institute of Technology, Atlanta, GA 30332. Email: xhe399@gatech.edu. Research supported by NSF Award DMS-2103154.} \and
Dhruv Mubayi\thanks{Department of Mathematics, Statistics and Computer Science, University of Illinois, Chicago, IL 60607. Email: mubayi@uic.edu. Research supported by NSF Awards DMS-1763317,
DMS-1952767 and DMS-2153576, by a Humboldt Research Award and by a Simons Fellowship.} \and
Andrew Suk\thanks{Department of Mathematics, University of California at San Diego, La Jolla, CA 92093. Email: asuk@ucsd.edu. Research supported by an NSF
CAREER Award and by NSF Awards DMS-1952786 and DMS-2246847.} \and 
Jacques Verstra\"ete\thanks{Department of Mathematics, University of California at San Diego, La Jolla, CA 92093. Email: jacques@ucsd.edu. Research supported by NSF Award DMS-1800332.} \and
Hung-Hsun Hans Yu\thanks{Department of Mathematics, Princeton University, Princeton, NJ 08544. Email: hansonyu@princeton.edu. }}
\date{}

\begin{document}

\maketitle

\begin{abstract}
A natural open problem in Ramsey theory is to determine those $3$-graphs $H$ for which the off-diagonal Ramsey number $r(H, K_n^{(3)})$ grows polynomially with $n$. We make substantial progress on this question by showing that if $H$ is tightly connected or has at most two tight components, then $r(H, K_n^{(3)})$ grows polynomially if and only if $H$ is contained in an iterated blowup of an edge.
\end{abstract}

\section{Introduction}

Given a $k$-uniform hypergraph $H$ (henceforth, \textit{$k$-graph}), the \emph{off-diagonal Ramsey number} $r(H, K_n^{(k)})$ is the smallest natural number $N$ such that every red/blue-coloring of the edges of $K_N^{(k)}$, the complete $k$-graph with $N$ vertices, contains either a red copy of $H$ or a blue copy of $K_n^{(k)}$. For graphs, the $k = 2$ case, we know that $r(H, K_n)$ always grows polynomially with $n$ and the main problem is to determine the growth rate more precisely. 
This problem remains open even when $H$ is a clique, where the correct polynomial dependency is only understood for $K_3$ and $K_4$ --- for $K_3$ the off-diagonal Ramsey number was famously determined up to a constant factor by Ajtai, Koml\'os and Szemer\'edi~\cite{AKS} and Kim~\cite{Ki}, while for $K_4$ a recent result of Mattheus and Verstra\"ete~\cite{MaV} shows that $r(K_4, K_n) = n^{3+o(1)}$.
When $H$ is a cycle, even less is known and determining the polynomial order of $r(C_4,K_n)$ in $n$ is a major Erd\H{o}s problem (see \cite{BK, CLRZ00} for the best bounds on this problem and \cite{CMMV, MV} for recent progress on other cycle-complete Ramsey numbers).

For $3$-graphs $H$, $r(H, K_n^{(3)})$ does not always grow polynomially. Indeed, it is already the case~\cite{CFS} that $r(K_4^{(3)}, K_n^{(3)}) \geq 2^{\Omega(n\log n)}$. Much of the recent work on off-diagonal hypergraph Ramsey numbers has focused on extending this lower bound to other hypergraphs, so that it is now known that $r(H, K_n^{(3)}) \geq 2^{\Omega(n\log n)}$ for $K_4^{(3)}\setminus e$ and, more generally, all links of odd cycles~\cite{FH} and for all tight cycles of length not divisible by three~\cite{CFG+}. 

Here we look in a different direction, our concern being with the problem of classifying those $H$ for which $r(H, K_n^{(3)})$ grows polynomially in $n$. Although this problem has been studied since at least work of Erd\H{o}s and Hajnal~\cite{EH} in the early 1970s, it seems to have been first raised explicitly by the first author~\cite{CAIM} at an AIM workshop in 2015. We propose a full classification, as follows.

\begin{conj} \label{conj:main}
For a $3$-graph $H$, there exists a constant $c$ depending only on $H$ such that $r(H, K_n^{(3)}) \leq n^c$ for all $n$ if and only if $H$ is a subgraph of an iterated blowup of an edge. 
\end{conj}

To clarify what we mean by an iterated blowup, we first note that a \emph{blowup of an edge} is simply a complete tripartite $3$-graph. An \emph{iterated blowup of an edge} is then any graph which is either a blowup of an edge or formed iteratively by placing another iterated blowup in one or more of the parts in a blowup of an edge. In what follows, we say that a $3$-graph is \emph{iterated tripartite} if it is contained in an iterated blowup of an edge. 

The direction of Conjecture~\ref{conj:main} saying that there exists a constant $c$ such that $r(H, K_n^{(3)}) \leq n^c$ if $H$ is iterated tripartite was already known to Erd\H{o}s and Hajnal~\cite{EH} (see also~\cite{CFS, FH}). Their concern was with a slightly different Ramsey-type question.
For natural numbers $4 \le s < n$ and $2 \le t \le \binom{s}{3}$, they were interested in the Ramsey function $r_3(s, t; n)$, the smallest natural number $N$ such that every red/blue-coloring of the edges of $K_N^{(3)}$ contains either a blue copy of $K_n^{(3)}$ or a set of $s$ vertices with at least $t$ red edges. Regarding the behavior of this function, they conjectured that there should be a polynomial-to-exponential transition for the growth rate of $r_3(s, t; n)$ at $t = t(s)$, the maximum number of edges in an iterated tripartite $3$-graph with $s$ vertices. Their conjecture may be seen as a toy model for our 
Conjecture~\ref{conj:main}. 

In their paper, Erd\H{o}s and Hajnal~\cite{EH} proved one direction of their conjecture, showing that $r_3(s, t; n)$ grows polynomially in $n$ for $t \leq t(s)$. Their proof, as indicated above, further shows one direction of our Conjecture~\ref{conj:main}. However, it remains open to show that $r_3(s, t; n)$ grows exponentially in a power of $n$ for $t > t(s)$. Some partial results, saying, for instance, that the conjecture holds when $s$ is a power of $3$, were proven by Conlon, Fox and Sudakov~\cite{CFS}, while the analogous problem in higher uniformities was solved completely by Mubayi and Razborov~\cite{MRa}. In particular, we see that Conjecture~\ref{conj:main} holds if $|H|$ is a power of $3$ and the number of edges in $H$ is larger than that in any iterated blowup with $|H|$ vertices. 

In light of these results, it is perhaps surprising that Conjecture~\ref{conj:main} has not been stated before. One reason for this was the common belief in the community that $r(H, K_n^{(3)})$ should also be polynomial for linear $H$, where a hypergraph is said to be \emph{linear} if any two edges in the hypergraph share at most one vertex. However, it was recently shown~\cite{CFG+} that this belief is mistaken and that there are linear hypergraphs for which $r(H, K_n^{(3)})$ grows superpolynomially. Conjecture~\ref{conj:main} is stronger again, suggesting that linearity is a red herring in this context.

Towards Conjecture~\ref{conj:main}, we prove two results. We say that a $3$-graph $H$ is \emph{tightly connected} if, for any two edges $e$ and $f$ of $H$ there exists a sequence of edges $e = e_0, e_1, \dots, e_t = f$ such that $e_{i-1}$ and $e_i$ share two vertices for all $i = 1, \dots, t$. Our first result says that Conjecture~\ref{conj:main} holds for tightly connected hypergraphs, even giving a lower bound in this case which is exponential in a power of $n$.

\begin{thm} \label{thm:tight}
If $H$ is a $3$-graph which is tightly connected and not tripartite, then $r(H, K_n^{(3)}) \ge 2^{\Omega(n^{2/3})}$.
\end{thm}

One objection towards this being strong evidence for Conjecture~\ref{conj:main} is that a tightly connected hypergraph is almost exactly the opposite of a linear hypergraph. Our second result goes some small way towards overruling this objection. We define a \textit{tight component} of $H$ to be a maximal tightly connected subgraph and observe that the edge set of every $3$-graph decomposes into disjoint tight components. We prove that Conjecture~\ref{conj:main} holds for hypergraphs which have at most two tight components, though our lower bound is considerably weaker in this case.

\begin{thm}\label{thm:two-components}
    If $H$ is a $3$-graph with at most two tight components and not iterated tripartite, then $r(H,K_n^{(3)}) \ge 2^{\Omega(\log^2 n)}$.
\end{thm}

The proof of Theorem~\ref{thm:two-components} is significantly more technical than that of Theorem~\ref{thm:tight} and may be considered our main result. With that in mind, we will warm up by proving Theorem~\ref{thm:tight} in the next section, before returning to Theorem~\ref{thm:two-components} in Section~\ref{sec:two-components}. We then conclude with some further remarks and open problems.

\section{Tightly connected hypergraphs} \label{sec:tight}

In this short section, we prove Theorem~\ref{thm:tight}, the statement that if $H$ is tightly connected and not tripartite, then $r(H, K_n^{(3)}) \ge 2^{\Omega(n^{2/3})}$.
In fact, we will  prove a stronger statement, from which \cref{thm:tight} clearly follows.

\begin{thm}\label{thm:all-tight}
    For every positive integer $N$, there is a red/blue edge coloring of $K_N^{(3)}$ vertices such that any red tightly connected subgraph is tripartite and the largest blue clique has order $O((\log N)^{3/2})$. 
\end{thm}

\begin{proof}[Proof of Theorem~\ref{thm:all-tight}] Let $\ell= C\log N$, where $C>0$ is a sufficiently large absolute constant, and $n = \ell^{3/2}$. Let $r=0.01\ell$ and let $V$ be an $r$-trifference code in $\{1, 2, 3\}^\ell$ of size $N$. What this means is that $V$ is a subset of $\{1, 2, 3\}^\ell$ of size $N$ with the property that, for any triple  $uvw$ of elements of $V$, there are at least $r$ coordinates where $\{u_i, v_i, w_i\} = \{1, 2, 3\}$. Such a set of size $N$ can be found by the first moment method, as it is exponentially unlikely that a random triple $uvw$ is not $r$-trifferent.

    We will define our Ramsey coloring on the complete $3$-graph with vertex set $V$. Let $c(uv)$ denote the set of coordinates in which $u,v$ differ. By the definition of $V$, $r\le |c(uv)| \le \ell$ for all $u,v$. Define $\phi(uv)$ to be an element of $c(uv)$ picked uniformly at random. The Ramsey coloring of $K_N^{(3)}$ will be $\chi$ where $\chi(uvw)$ is red if and only if $\phi(uv)=\phi(vw)=\phi(uw)$. Observe that, by definition, any red triple $uvw$ automatically satisfies $\{u_i, v_i, w_i\} = \{1, 2, 3\}$ for $i=\phi(uv)$.

    \medskip
    \noindent
    \textbf{Red tight components are tripartite.} For any red tight component, all pairs in its $2$-shadow must share the same $\phi$-value $i$. Then the $i$-th coordinates of the vertices give a tripartition of the vertices, showing that the component is tripartite.

    \medskip
    \noindent
    \textbf{No blue $K_n^{(3)}$ with positive probability.} Fix a set $U$ of $n$ vertices in $V$; we would like to bound the probability that this $n$-set forms a blue clique. Let $T_{uvw}$ be the event that the triple $uvw$ is red, that is, that the triangle $uvw$ is monochromatic under $\phi$.
    Each event has probability 
    \[
    \Pr[T_{uvw}] = \frac{|c(uv)\cap c(vw)\cap c(uw)|}{|c(uv)|\cdot |c(vw)|\cdot |c(uw)|} = \Theta(\ell^{-2})
    \]
    by our choice of $V$ and $c(\cdot)$.
    Since there are $\Theta(n^3)$ total events, the expected number of monochromatic triangles in $U$ is $\mu = \Theta(n^3/ \ell^2) = \Theta(\ell^{5/2})$. The Poisson paradigm predicts that
    \begin{equation}\label{eq:poisson}
    \Pr[\bigwedge \overline T_{uvw}] \le e^{-\Theta(\ell^{5/2})}.
    \end{equation}
    We first show that this would suffice to prove the theorem. Indeed, since there are at most $N^n$ $n$-sets and
    \[
    N^n e^{-\Theta(\ell^{5/2})} < 1
    \]
    by our choices of $n, \ell$ and $N$, the union bound gives that there is a positive probability no blue $n$-cliques appear.

    Now we prove \eqref{eq:poisson} via Suen's inequality in the form given in Alon--Spencer~\cite[Theorem 8.7.1]{AlSp}, which implies that 
    \begin{equation}\label{eq:suen}
    \Pr[\bigwedge \overline{T_{\tau}} ] \le e^{-\mu + \sum_{\tau\sim \tau'} y(\tau, \tau')},
    \end{equation}
    where $\tau$ ranges over all triangles $uvw$ in $U$ and $\tau \sim \tau'$ if these two triangles share at least one edge, with
    \[
    y(\tau, \tau') = (\Pr[T_{\tau} \wedge T_{\tau'}] + \Pr[T_{\tau}]\cdot \Pr[T_{\tau'}])\prod_{\tau'' \sim \tau \textnormal{ or }\tau'' \sim \tau'} (1- \Pr[T_{\tau''}])^{-1}.
    \]
    Observe that if $\tau$ and $\tau'$ share an edge, then $\Pr[T_{\tau} \wedge T_{\tau'}] = O(\ell^{-4})$ and $\Pr[T_{\tau}]\cdot \Pr[T_{\tau'}] = O(\ell^{-4})$. Moreover, there are $\Theta(n)$ triangles $\tau''$ that share an edge with either $\tau$ or $\tau'$, so we obtain
    \[
    y(\tau,\tau') =O(\ell^{-4}) (1-\Theta(\ell^{-2}))^{-\Theta(n)}=O(\ell^{-4}) e^{\Theta(n/\ell^2)} = O(\ell^{-4}),
    \]
    where we used that $n=o(\ell^2)$. Since there are $O(n^4)$ pairs of triangles $\tau,\tau'$ that share an edge, \eqref{eq:suen} reduces to
    \[
    \Pr[\bigwedge \overline{T_\tau}] \le e^{-\mu + O(n^4/\ell^4)} = e^{-\Theta(\ell^{5/2}) + O(\ell^2)} = e^{-\Theta(\ell^{5/2})},
    \]
    as desired.
\end{proof}

Before moving on to the case with two tight components, we sketch an alternative, but quantitatively weaker, construction for the single component case.
Let $\ell = C\log N$ and $n=C^{7/4} \log^2 N$ for $C$ a sufficiently large constant. For each edge $uv$ of the complete graph on $V = [N]$, label it with some $\phi(uv) \in \{1,2,..., \ell\}$ chosen independently and uniformly at random. Moreover, assign a string $f(v) \in \{1,2,3\}^\ell$ to each vertex $v$, again chosen independently and uniformly at random. The Ramsey coloring of $K_N^{(3)}$ will be $\chi$ where $\chi(uvw)$ is red if and only if
\begin{enumerate}
    \item $\phi(uv)=\phi(vw)=\phi(uw) = i$ for some $i$,
    \item $f_i(u), f_i(v)$ and $f_i(w)$ are all distinct for this $i$
\end{enumerate}
and blue otherwise. If there is any red tightly connected $3$-graph in $\chi$, then all edges in its $2$-shadow have the same $\phi$-value $i$, so the $f_i(v)$ yield a tripartition. To show that there is no blue $K_n$ with positive probability, we use an application of Janson's inequality to show that, with positive probability, for every subset $U$ of $V$ with $n$ vertices and every $i \in \{1,2,..., \ell\}$ there are at least $n/6$ vertex-disjoint triangles in $U$ which are monochromatic of color $i$ under $\phi$. If we now bring the randomness of $f$ into play, then each such monochromatic triangle forms a blue edge with probability $7/9$. As the randomness for each coordinate $f_i$ of $f$ is independent, the induced hypergraph on any given $n$-vertex set $U$ is completely blue with probability at most $(7/9)^{\frac{n}{6} \cdot \ell} = \exp(-\Theta(n\ell))$. Since there are ${N \choose n} = \exp(O(n\log N))$ subsets to consider, the union bound and the fact that $\ell$ dominates $\log N$ completes the analysis, yielding the bound $r(H, K_n^{(3)}) \geq 2^{\Omega(n^{1/2})}$ for every non-tripartite tightly connected $3$-graph $H$.

\section{Two tight components} \label{sec:two-components}

In this section, we prove Theorem~\ref{thm:two-components}, which says that if $H$ has at most two tight components and is not iterated tripartite, then $r(H,K_n^{(3)}) \ge 2^{\Omega(\log^2 n)}$.
We will once again prove \cref{thm:two-components} by providing a single construction avoiding all such $H$ in red simultaneously.

\begin{thm}\label{thm:all-two-components}
    For every positive integer $N$, there is a red/blue edge coloring of $K_N^{(3)}$ such that any red subgraph with at most two tight components is iterated tripartite and the largest blue clique has order $\exp(O(\sqrt{\log N}))$. 
\end{thm}

For the proof, we will need to show the existence of an edge coloring of the complete graph with certain powerful properties, encapsulated in the following lemma.

\begin{lemma}\label{lem:rainbow-coloring}
    For positive integers $A\geq 20$, $\ell \ge 1$ and $N = 2^\ell$, there exists an edge coloring $\varphi : E(K_N) \rightarrow [\ell A]$ such that:
\begin{enumerate}
    \item For each color, the edges with that color form a vertex-disjoint union of bicliques.
    \item In every $k$-vertex subset with $k=\exp\left(\Omega(\log N/\log A)\right)$, there are $\Omega(k^{2.5})$ rainbow triangles.
\end{enumerate}
\end{lemma}

We will postpone the proof of this lemma for now, first showing how to use it to prove Theorem~\ref{thm:two-components}.

\begin{proof}[Proof of Theorem \ref{thm:all-two-components}]
Replace $N$ by the smallest power of $2$ exceeding $N$ and let $\ell = \log_2N$.
Let $\varphi$ be as in Lemma~\ref{lem:rainbow-coloring} with $A, k = \exp(\Theta(\sqrt{\ell}))$. Let $g, f_1, f_2, f_3$ be four independent random functions on ${[N] \choose 2}$, where $g: {[N] \choose 2} \rightarrow \{1, 2, 3\}$ uniformly at random and $f_1, f_2, f_3 : {[N] \choose 2} \rightarrow [\ell A]$ uniformly at random.
The Ramsey coloring of the complete $3$-graph on $[N]$ will be $\chi$ where $\chi(uvw)$ for a triple $uvw$ with $u < v < w$ is red if and only if 
\begin{enumerate}
    \item $uvw$ is a rainbow triangle with respect to $\varphi$,
    \item $g(uv) = 1$, $g(vw) = 2$ and $g(uw) = 3$,
    \item $f_1 (uw) = f_1 (vw) = \varphi(uv)$,
    \item $f_2 (uv) = f_2 (uw) = \varphi(vw)$,
    \item $f_3 (uv) = f_3 (vw) = \varphi(uw)$.
\end{enumerate}
Otherwise, the triple $uvw$ is colored blue.

\paragraph{No large blue cliques.}
We upper bound the probability that some subset $S$ of order $k$ forms a blue clique.
Provided $k= \exp(\Omega(\ell/\log A))$, by the properties of $\varphi$, there are $\Omega(k^{2.5})$ rainbow triangles in $S$. By greedily picking out triangles, we may find $\Theta(k^{1.5})$ of them that are edge-disjoint.
Each such triangle is colored red independently with probability $\Theta((\ell A)^{-6})$.
Therefore, the probability that none of them is colored red is $e^{-\Theta(k^{1.5}(\ell A)^{-6})}.$
By the union bound, the probability that there is a blue clique can then be bounded by
\[\binom{N}{k}e^{-\Theta(k^{1.5}(\ell A)^{-6})} = e^{\Theta(k\log N)-\Theta(k^{1.5}(\ell A)^{-6})},\]
which is less than $1$ if $k = \Omega(\ell^{14}A^{12})$. 
To balance this with the condition that $k= \exp(\Omega(\ell/\log A))$,
it suffices to pick $A, k = \exp(\Theta(\sqrt{\ell}))$, as indicated at the outset.

\paragraph{Union of two red tight components is iterated tripartite.}
We now show that the union of any two red tight components is iterated tripartite. We start by proving a structural result for each red tight component.
\begin{lemma}
    Every red tight component $\cC(V_{\cC}, E_{\cC})$ is tripartite and there is a unique tripartition $V^{(1)}_{\cC}\cup V^{(2)}_{\cC}\cup V^{(3)}_{\cC}$ of the vertices $V_{\cC}$.
    Moreover, there exist three colors $c^{(1)}_{\cC}$, $c^{(2)}_{\cC}$ and $c^{(3)}_{\cC}$ such that, for any ordering $i,j,k$ of $1,2,3$, $\varphi(uv) = c^{(k)}_{\cC}$ for any $u\in V^{(i)}_{\cC}$ and $v\in V^{(j)}_{\cC}$.
\end{lemma}
\begin{proof}
    For any hyperedge $\{x,y,z\}$, define $\Phi(\{x,y,z\})$ to be the set $\{\varphi(xy), \varphi(yz),\varphi(xz)\}$. 
    For any two red hyperedges $e,e'$ sharing two vertices $x,y$, if $e = \{x,y,z\}$ and $e'=\{x,y,z'\}$, then the relative order of $x,y,z$ and the relative order of $x,y,z'$ have to be the same, showing that $g(xz') = g(xz)$ and $g(yz') = g(yz)$.
    Therefore,
    \[\varphi(xz) = f_{g(xz)}(xy) = f_{g(xz')}(xy) = \varphi(xz')\]
    and similarly $\varphi(yz) = \varphi(yz')$.
    Hence, as $\cC$ is tightly connected in red, $\Phi(e)$ is the same for any $e\in E_{\cC}$.  

    Now pick an arbitrary $e\in E_{\cC}$ and let $c^{(1)}_{\cC}, c^{(2)}_{\cC}, c^{(3)}_{\cC}$ be an arbitrary enumeration of the elements in $\Phi(e)$.
    Moreover, let $G$ be the subgraph of $K_N$ obtained by taking the $2$-shadow of $\cC$.
    By induction, it is easy to see that the edges of color $c^{(i)}_{\cC}$ in $G$ are connected for each $i\in[3]$.

    Now, for any $v\in V_{\cC}$, choose arbitrary vertices $u,w$ such that $\{u,v,w\}\in E_{\cC}$ and define $i(v)$ to be such that $\varphi(uw) = c^{i(v)}_{\cC}$.
    We show that $i(v)$ does not depend on the choice of $u$ and $w$.
    Suppose instead that there are some other $u',w'$ with $\{u',v,w'\}\in E_{\cC}$ such that $\varphi(uw) \neq  \varphi(u'w')$. Then we know that one of $\varphi(u'v)$ and $\varphi(vw')$ is equal to $\varphi(uw)$.
    Thus, in $G$, $u,v,w$ all belong to the connected subgraph of color $\varphi(uw)$ with $u$ and $w$ neighbors.
    Therefore, without loss of generality, we can assume that there is a monochromatic odd walk from $v$ to $u$ of color $\varphi(uw)$.
    Since the edges with color $\varphi(uw)$ form a vertex-disjoint union of bicliques, it must then be the case that $\varphi(uv) = \varphi(uw)$.
    But this is a contradiction, so $i(v)$ must be independent of the choice of $u$ and $w$.

    We partition $V_{\cC}$ into $V^{(1)}_{\cC}, V^{(2)}_{\cC}$ and $V^{(3)}_{\cC}$ based on the labels $i(v)$.
    It is clear that $\cC$ is tripartite with respect to this tripartition and, as $\cC$ is tightly connected, that this is the only possible tripartition.
    Finally, if $i,j,k$ is an ordering of $1,2,3$, then it is also clear that the connected subgraph of $G$ of color $c^{(k)}_{\cC}$ is a bipartite graph on $V^{(i)}_{\cC}\cup V^{(j)}_{\cC}$.
    If $u\in V^{(i)}_{\cC}$ and $v\in V^{(j)}_{\cC}$, then, by the fact that the edges of color $c^{(k)}_{\cC}$ form a vertex-disjoint union of bicliques in $K_N$, we know that $\varphi(uv) = c^{(k)}_{\cC}$, as desired.
\end{proof}

Suppose now that $\cC$ and $\cC'$ are two red tight components whose union is not tripartite.
Without loss of generality, this means that there is some $u\in V^{(1)}_{\cC}\cap V^{(1)}_{\cC'}$ and $v\in V^{(2)}_{\cC}\cap V^{(1)}_{\cC'}$.
If some $w\in V^{(3)}_{\cC}$ is in $V^{(2)}_{\cC'}\cup V^{(3)}_{\cC'}$, then, since every edge between $V^{(1)}_{\cC'}$ and $V^{(i)}_{\cC'}$ has the same color under $\varphi$ for each of $i = 2, 3$, $c^{(2)}_{\cC} = \varphi(uw) =\varphi(vw)=c^{(1)}_{\cC}$, which is a contradiction.
Therefore, $V^{(3)}_{\cC}\cap (V^{(2)}_{\cC'}\cup V^{(3)}_{\cC'})$ must be empty.
If there is some $w\in V^{(1)}_{\cC}$ that is in $V^{(2)}_{\cC'}\cup V^{(3)}_{\cC'}$, say $w\in V^{(2)}_{\cC'}$,
then, since every edge between $V^{(1)}_{\cC}$ and $V^{(2)}_{\cC}$ has the same color under $\varphi$, $\varphi(uv) = c^{(3)}_{\cC} = \varphi(vw) = c^{(3)}_{\cC'}$. But then picking any $x\in V^{(2)}_{\cC'}$ we have a monochromatic triangle $uvx$, which contradicts the fact that each color class in $\varphi$ is bipartite. 
Therefore, $V^{(1)}_{\cC}\cap (V^{(2)}_{\cC'}\cup V^{(3)}_{\cC'})$ must also be empty.
Similarly, $V^{(2)}_{\cC}\cap (V^{(2)}_{\cC'}\cup V^{(3)}_{\cC'})$ must be empty.
Hence, $V^{(2)}_{\cC'}\cup V^{(3)}_{\cC'}$ and $V_{\cC}$ are disjoint, so the union of $\cC$ and $\cC'$ is iterated tripartite, as required.
\end{proof}

It remains to prove Lemma~\ref{lem:rainbow-coloring}. We first record a technical result about binary trees. To state this result, we define the \emph{weight} of an internal vertex in a binary tree to be $2$ less than the number of leaves in the subtree consisting of this vertex and all its descendants.

\begin{lemma} \label{lem:tree}
    There exists a constant $C$ such that, for any binary tree $T$ with $k$ leaves and any collection $X$ of vertices of $T$ with total weight less than $k\log k/C$, the number of triples of leaves whose least common ancestor lies outside $X$ is at least $k^{2.5}$.
\end{lemma}
\begin{proof}
For any vertex $v$, let $n_v$ be the number of leaves in the subtree rooted at $v$ and let $m_v$ be the number of triples of leaves whose least common ancestor is $v$. 
Define an \emph{imbalance} in a tree $T$ to be a configuration where $u$ is the parent of $v$ and $w$, $v$ is the parent of $x$ and $y$, $n_x>n_w$ and $n_x\geq n_y$. 
In this case, we say that the imbalance is at $u$.
The \emph{rotation} of the imbalance is the modified  tree $T'$ where the subtree rooted at $x$ is swapped with the smaller subtree rooted at $w$.
Note that the rotation increases the depths of leaves below $w$ by $1$ and decreases the depths of the leaves below $x$ by $1$.
If we set $D_{\textup{total}}$ to be the sum of the depths of the leaves in the binary tree, then the rotation at $u$ decreases $D_{\textup{total}}$ by $n_x-n_w>0$.

Let $V_{\textup{in}}$ be the set of internal vertices of $T$.
For any function $f:V_{\textup{in}}\to [0,1]$ that maps the internal vertices to real numbers in $[0,1]$, we define its \emph{weight} by $\sum_{v\in V_{\textup{in}}}f(v)\cdot (n_v-2)$ and its \emph{score} by $\sum_{v\in V_{\textup{in}}}(1-f(v))\cdot m_v$.
To complete the proof of the lemma, it will suffice to show that there exists a constant $C$ such that any function $f:V_{\textup{in}}\to [0,1]$ with weight less than $k\log k/C$ has score at least $k^{2.5}$. 

For any $W\geq 0$, let $\nu^*_T(W)$ be the minimum score of a function $f:V_{\textup{in}}\to [0,1]$ with weight at most $W$.
If $T$ has an imbalance labeled with $u,v,w,x,y$ as in the definition, let $T'$ be the rotation of the imbalance.
Then the only changes in weights or scores occur at $u$ and $v$.
The following table records how the weights and scores change at these two vertices:

\begin{center}
\begin{tabular}{c|c|c}
     & $T$ (before rotation)  & $T'$ (after rotation) \\\hline
   $n_u$  & $n_u=(n_x+n_y+n_w)-2$ & $n_u'=(n_x+n_y+n_w)-2=n_u$\\\hline
   $n_v$ & $n_v=(n_x+n_y)-2$ & $n_v'=(n_y+n_w)-2$ \\\hline
   $m_u$ & $m_u=\frac{1}{2}(n_x+n_y)n_w[(n_x+n_y+n_w)-2]$ & $m_u'=\frac{1}{2}n_x(n_y+n_w)[(n_x+n_y+n_w)-2]$\\\hline
   $m_v$ & $m_v=\frac{1}{2}n_x n_y[(n_x+n_y)-2]$ & $m_v'=\frac{1}{2}n_yn_w[(n_y+n_w)-2]$
\end{tabular}
\end{center}

Let $f$ be the optimizer for $\nu^*_T(W)$.
Let $W' = f(u)(n_u-2)+f(v)(n_v-2)\in [0,n_u+n_v-4]$ be the total contribution of $u$ and $v$ to the weight of $W$.
By the optimality of $f$, we know that $(f(u),f(v))$ is the maximizer of the following linear program:
\begin{align*}
    \textup{maximize }& m_ua+m_vb\\
    \textup{subject to }& \begin{cases}
        0\leq a,b\leq 1,\\
        (n_u-2)a+(n_v-2)b\leq W'.
    \end{cases}
\end{align*}
Denote by $M(W')$ the optimum of this linear program.
If we set 
\[(t,t') = \begin{cases}(u,v)&\textup{ if }\frac{m_u}{n_u-2}\geq \frac{m_v}{n_v-2},\\(v,u)&\textup{ otherwise},\end{cases}\] 
then $M(W')$ is a line through the origin with slope $m_t/(n_t-2)$ on $[0,n_t-2]$ and the slope becomes $m_{t'}/(n_{t'}-2)$ on $[n_{t}-2,n_u+n_v-4]$.

We may similarly define $M'(W')$, where we maximize $m_u'a+m_v'b$ subject to the conditions $a,b\in [0,1]$ and $(n_u'-2)a+(n_v'-2)b\leq W'$.
Note now that 
\[\frac{m_u'}{n_u'-2} = \frac 12 (n_y+n_w)n_x> \frac 12 n_xn_y> \frac 12n_yn_w = \frac{m_v'}{n_v'-2},\]
so $M'(W')$ is a line through the origin with slope $m_u'/(n_u'-2)$ on $[0,n_u'-2]$ and the slope becomes $m_v'/(n_v'-2)$ on $[n'_u - 2, n_u'+n_v'-4]$, after which it remains constant. 

Given the definitions of the functions $M(W')$ and $M'(W')$, we make the following claim.

\begin{claim}
    For all $W'\in [0,n_u+n_v-4]$, $M'(W')\geq M(W')$.
\end{claim}
\begin{proof}
    As $M'(W')$ is concave and $M(W')$ is piecewise linear, it suffices to verify this for $W'=0$, $W'=n_t-2$ and $W'=n_u+n_v-4$.
    The inequality clearly holds at $W'=0$.
    Moreover, since $n_t-2\leq n_u-2=n_u'-2$, we know that
    \[M'\left(n_t-2\right)-M(n_t-2) = \frac{m_u'}{n_u'-2}(n_t-2)-\frac{m_t}{n_t-2}(n_t-2) = \left(\frac{m_u'}{n_u'-2}-\frac{m_t}{n_t-2}\right)(n_t-2).\]
    Hence, to verify the inequality for $W'=n_t-2$, it suffices to show that $m_u'/(n_u'-2)\geq m_t/(n_t-2)$.
    Computing directly, we see that 
    \[\frac{m_u'}{n_u'-2} = \frac 12(n_y+n_w)n_x > \frac 12(n_x+n_y)n_w = \frac{m_u}{n_u-2}\]
    and
    \[\frac{m_u'}{n_u'-2} = \frac 12(n_y+n_w)n_x > \frac 12 n_xn_y = \frac{m_v}{n_v-2}.\]
    Finally, to verify the inequality for $W'=n_u+n_v-4$, note that $M(n_u+n_v-4)$, $M'(n_u'+n_v'-4)$ are equal as they both count the triples of leaves which meet at least two of the subtrees rooted at $x$, $y$ and $w$.
    Since $n_u'+n_v'-4<n_u+n_v-4$, we have $M'(n_u+n_v-4)=M'(n_u'+n_v'-4)=M(n_u+n_v-4)$, as desired.   
\end{proof}

If we now set $f'$ to be the same as $f$ except that $f'(u)=a$, $f'(v)=b$, where $(a,b)$ is the extremizer for $M'(W')$, then both the weight and the score of $f'$ in $T'$ are at most those of $f$ in $T$, so that $\nu_{T'}^*(W)\leq \nu_T^*(W)$. That is, rotating imbalances does not increase $\nu_T^*(W)$, so, since we are trying to give a lower bound on $\nu_T^*(W)$, we may rotate as often as we please.

To apply this observation, we need to show that for any binary tree $T$ it is possible to do a series of rotations so that the resulting binary tree has no remaining imbalances.
But, as observed earlier, any rotation strictly decreases $D_{\textup{total}}$, the sum of the depths of the leaves.
Since $D_{\textup{total}}$ is always a non-negative integer, this shows that the binary tree will have no imbalances after finitely many rotations.

We may therefore assume that $T$ has no imbalances. 
In particular, for any vertex $u$ with children $v$ and $w$, we must have that $n_v, n_w$ are both in $[n_u/3, 2n_u/3]$.
Otherwise, without loss of generality, we may assume that $n_w<n_u/3$ and $n_v>2n_u/3$.
But then there would exist a child $x$ of $v$ such that $n_x>n_u/3>n_w$, contradicting that there are no imbalances. 
Therefore, any vertex $u$ at depth $d<0.1\log k$ is an internal vertex with weight $\Omega(3^{-d}k)$ and $m_u=\Omega(3^{-3d}k^3)$. 
Since $\Omega(3^{-3d}k^3) = \omega(k^{2.5})$, if the score of a function $f:V_{\textup{in}}\to [0,1]$ is less than $k^{2.5}$, then any vertex whose depth is less than $0.1\log k$ must be mapped to a number which is at least $1/2$ by $f$.
Since the vertices at depth $d$ have total weight $k-O(2^d)$ for every $d<0.1\log k$, it follows that the weight of $f$ is $\Omega(k\log k)$, as desired.
\end{proof}
Now we are ready for the proof of the key lemma.

\begin{proof}[Proof of Lemma~\ref{lem:rainbow-coloring}.]
Let $N = 2^{\ell}$ and let $G$ be a complete graph on vertex set $[N]$. 
For each $i\in [\ell]$ and $u\in [N]$, let $c_i(u)$ be an independent uniform element of $[A]$.
Finally, for any $u\neq v\in [N]$, if $t=v_2(u-v)$, where $v_2(x)$ is the $2$-adic valuation function whose value is the highest power of $2$ dividing $x$, we color the edge $uv$ with the product color $(t, (-1)^{\lfloor u/2^t\rfloor} (c_t(u)-c_t(v)) \bmod A)$.
Note that as $\lfloor u/2^t\rfloor$ and $\lfloor v/2^t\rfloor$ have different parities, the color of $uv$ does not depend on how we order $u$ and $v$. Moreover, each color class is the union of vertex-disjoint bicliques, as required.

Let $k$ be a positive integer to be determined later and let $S\subseteq [N]$ be any subset of order $k$.
We wish to upper bound the probability that there are too few rainbow triangles on this set of vertices. 
To this end, consider the following binary tree $T$ induced by $S$: place $S$ at the root vertex; consider the first bit where some elements of $S$ do not agree when written in binary and then split $S = S_0\cup S_1$ based on the value of this bit, where we assume $\abs{S_0}\geq \abs{S_1}$; attach the binary trees induced by $S_0$ and $S_1$ to the root vertex.

For an internal vertex corresponding to the set $S'$, let $S_0'$ and $S_1'$ be the two subsets corresponding to its children with $\abs{S_0'}\geq \abs{S_1'}$.
Suppose that the first bit where some elements disagree is the $t$-th bit.
Call $S'$ \emph{good} if there is no color that appears more than $|S'_0|/2$ times among the collection of $c_t(u)$ with $u \in S_0'$  
and \emph{bad} otherwise.
If $\abs{S'}\geq 3$, then $\abs{S_0'}\geq 2$, so the probability that $S'$ is bad can be bounded above by 
$$A \cdot 2^{|S'_0|} \cdot A^{-(|S'_0|+1)/2} \leq \exp(-\Omega((\abs{S'}-2))\log A),$$
where we used that $A \ge 20$. 
Moreover, the event that $S'$ is bad is independent from all other similar events.

If $S'$ is good, then we know that, for every $u\in S_0'$, there are at least $\abs{S_0'}/2$ vertices $v \in S_0'$ with $c_t(u)\neq c_t(v)$. 
For any such $u,v\in S_0'$ and any $w\in S_1'$, it is clear that $uvw$ is rainbow.
Therefore, we have found at least
\[\frac 12 \frac{\abs{S_0'}^2}{2}\abs{S_1'}\geq \frac{1}{4}\left(\binom{\abs{S_0'}}{2}\abs{S_1'}+\binom{\abs{S_1'}}{2}\abs{S_0'}\right)\]
rainbow triangles with $S'$ as their least common ancestor.
Note that this is exactly a quarter of the number of triplets whose least common ancestor is $S'$.

If we now apply Lemma~\ref{lem:tree} with $X$ the set of bad vertices in $T$, we see that if the total weight on the vertices in $X$ is less than $k\log k/C$, then the number of triples of leaves whose least common ancestor is outside $X$, and therefore good, is at least $k^{2.5}$. By the observation above, this would mean that we have at least $k^{2.5}/4$ rainbow triangles in $S$. We may therefore assume that the total weight on the set of bad vertices is at least $k\log k/C$. Since the weight of the vertex corresponding to $S'$ is $\abs{S'}-2$ and the probability that $S'$ is bad can be bounded by $\exp(-\Omega((\abs{S'}-2))\log A)$, the probability that all of the vertices in $X$ are bad is at most $\exp(-\Omega(k\log k\log A))$. Since there are at most $2^k$ possible choices for $X$ (as there are exactly $k-1$ internal vertices), a union bound implies that the probability there are fewer than $k^{2.5}/4$ rainbow triangles in $S$ is at most $\exp(-\Omega(k\log k\log A))$.

To ensure, by a union bound, that there are at least $k^{2.5}/4$ rainbow triangles in all vertex subsets $S$ of order $k$ with positive probability, we need that 
\[\binom{N}{k}\exp(-\Omega(k\log k\log A))<1,\]
or, equivalently, that $k\log k\log A = \Omega(k\log N).$
Therefore, it suffices to take $k = \exp(\Omega(\log N/\log A))$.
\end{proof}

\section{Concluding remarks}

The main problem left open by this paper is whether Conjecture~\ref{conj:main} holds in full generality. However, it would already be interesting to prove it for hypergraphs with three tight components or to find a different proof for the two component case that gives a better bound. A particular case of interest is the Fano plane $F$, which is the unique $3$-graph with seven edges on seven vertices in which every pair of vertices is contained in a unique edge. The Fano plane is not iterated tripartite, so, according to Conjecture~\ref{conj:main}, $r(F, K_n^{(3)})$ should not grow polynomially. A proof of this would considerably strengthen our belief in the conjecture, on whose validity the authors do not form a completely united front.
Let us also point out that in trying to prove \cref{conj:main},
one cannot hope to prove a statement analogous 
to \cref{thm:all-tight} and \cref{thm:all-two-components} that gives a single Ramsey construction avoiding the entire family of possible $H$ simultaneously.
Indeed, if a red/blue edge coloring of $K_N^{(3)}$ has all its red subgraphs iterated tripartite, it is not hard to show by induction that there is a blue clique of size $2^{\log_3 N+O(1)} = \Omega(N^{\log_32}).$

Another problem of interest is to determine the growth rate of $r(\mathcal{H}, K_n^{(3)})$, where $\mathcal{H}$ is the family of all non-tripartite tightly connected $3$-graphs. 
\cref{thm:all-tight} gives the lower bound  $r(\mathcal{H}, K_n^{(3)}) \geq 2^{\Omega(n^{2/3})}$. As an upper bound, we can show that $r(\mathcal{H}, K_n^{(3)}) \leq 2^{O(n)}$. To see this, for each positive integer $N$, let $n(N)$ be the largest positive integer such that every red/blue-colored $K_N^{(3)}$ with no red copy of any hypergraph from $\mathcal{H}$ contains a blue $K_{n(N)}^{(3)}$. We claim that $n(N)\geq 1+n\left(\lceil\frac{N-1}{2}\rceil\right)$. Assuming this claim, a simple induction implies that $n(2^k-1)\geq k$ for all integers $k \geq 1$ and we are done.

To prove the claim, we pick an arbitrary vertex $v\in V(K_n^{(3)})$ and consider $G_v$, the red/blue-colored graph on $V \setminus \{v\}$ where the color of each edge $wx$ is the color of the triple $vwx$. 
Since, by assumption, all red tight components in our coloring of $K_n^{(3)}$ are tripartite, all red connected components of $G_v$ are bipartite.
This is exactly equivalent to saying that the set of red edges in $G_v$ is bipartite. 
Therefore, there is a subset $V'$ of $V\backslash \{v\}$ such that $G_v[V']$ contains only blue edges and $\abs{V'}\geq \lceil \frac{N-1}{2}\rceil$.
Now consider the coloring of the $3$-graph on $V'$, which again has no red copy of any hypergraph from $\mathcal{H}$.
By the definition of the function $n$, there exists a subset $V''$ of $V'$ of size $n\left(\lceil\frac{N-1}{2}\rceil\right)$ such that $V''$ is completely blue.
It is then easy to see that $V''\cup \{v\}$ is completely blue as well, as the presence of any red edge would contradict the choice of $V'$. The required bound on $n(N)$ follows.

While we suspect that the correct bound for $r(\mathcal{H}, K_n^{(3)})$ is of the form $2^{\Theta(n)}$, it would be more interesting if it were $2^{O(n^{1 - \epsilon})}$ for some $\epsilon > 0$. Such intermediate growth has been demonstrated~\cite{CFHMSV} for the Ramsey numbers $r(K_4^{(3)}, S_n^{(3)})$, where $S_n^{(3)}$ is the $3$-graph on $n+1$ vertices consisting of all $\binom{n}{2}$ edges incident to a given vertex, but it would be very interesting to have such an example with $K_n^{(3)}$ instead of $S_n^{(3)}$. A particular family of cases where such intermediate growth has not been ruled out is $r(C_\ell\setminus e, K_n^{(3)})$, where $C_\ell\setminus e$ is the tight cycle of length $\ell \not\equiv 0 \pmod{3}$ with a single edge removed. While Theorem~\ref{thm:tight} gives $r(C_\ell\setminus e, K_n^{(3)}) \geq 2^{\Omega(n^{2/3})}$, we do not currently know of any upper bound better than $r(C_\ell\setminus e, K_n^{(3)}) \leq r(C_\ell, K_n^{(3)}) \leq 2^{O_\ell(n \log n)}$ for $\ell$ sufficiently large~\cite{Mu}. 

It might be interesting to look at the analogous conjecture to Conjecture~\ref{conj:main} for higher uniformities. If Conjecture~\ref{conj:main} is indeed correct, its generalization should say that for a $k$-graph $H$, there exists a constant $c$ depending only on $H$ such that $r(H, K_n^{(k)}) \leq n^c$ for all $n$ if and only if $H$ is contained in an iterated blowup of an edge. One direction of this conjecture is simple, while certain partial results in the other direction again hold. For example, a straightforward extension of either construction in Section~\ref{sec:tight} implies that the conjecture holds for \emph{$2$-tightly connected hypergraphs}, hypergraphs where any two edges $e$ and $f$ are joined by a sequence of edges $e = e_0, e_1, \dots, e_t = f$ such that $e_{i-1}$ and $e_i$ share at least two vertices for all $i = 1, \dots, t$. In particular, if we generalize the argument that yields Theorem~\ref{thm:tight}, we obtain the following result.

\begin{thm}\label{thm:higher-uniformities}
    If $k\geq 3$ and $H$ is a $k$-graph which is 2-tightly connected and not $k$-partite, then $r(H, K_n^{(k)})\geq 2^{\Omega(n^{2/k})}$.
\end{thm}

One might also try to find necessary and sufficient conditions on $k$-graphs $H$ under which $r(H, K_n^{(k)})$ is upper bounded by a function of the form $2^{n^c}$ or $2^{2^{n^c}}$ and so on. In this direction, we make the following, likely difficult, conjecture, which would extend Theorem~\ref{thm:higher-uniformities}. Generalizing the definition above, we say that a hypergraph is \emph{$s$-tightly connected} if any two edges $e$ and $f$ are joined by a sequence of edges $e = e_0, e_1, \dots, e_t = f$ such that $e_{i-1}$ and $e_i$ share at least $s$ vertices for all $i = 1, \dots, t$.

\begin{conj}\label{conj:higher-uniformities}
    If $k > s$ and $H$ is a $k$-graph which is $s$-tightly connected and not $k$-partite, then there exists a positive constant $c$ such that $r(H, K_n^{(k)})\geq t_{s}(n^c)$, where the tower function $t_i(x)$ is defined by $t_1(x) = x$ and $t_{i}(x) = 2^{t_{i-1}(x)}$ for all $i \geq 2$.
\end{conj}

One natural way to prove \cref{conj:higher-uniformities} would be to consider the growth of the Ramsey number $r(\cH_s^{(k)}, K_n^{(k)})$, where $\cH_s^{(k)}$ is the family of $s$-tightly connected $k$-graphs that are not $k$-partite. We can make an even stronger conjecture than \cref{conj:higher-uniformities}.

\begin{conj}\label{conj:higher-uniformities-family}
   If $k > s$, then there exists a positive constant $c$ such that $r(\cH_s^{(k)}, K_n^{(k)})\geq t_s(n^c)$.
\end{conj}

Our methods prove this conjecture for $s=2$ and  
\cref{conj:higher-uniformities} follows immediately from \cref{conj:higher-uniformities-family} as $r(H, K_n^{(k)})\geq r(\cH_s^{(k)}, K_n^{(k)})$ for any $H\in \cH_s^{(k)}$.
We note that all the proofs in this paper followed this strategy.
Moreover, we observe that \cref{conj:higher-uniformities-family}, if true in general, would also determine the tower height of diagonal hypergraph Ramsey numbers, perhaps the central open problem in hypergraph Ramsey theory. Indeed, $r(K_n^{(s)}, K_n^{(s)})\geq r(\cH_s^{(k)}, K_n^{(k)})$ for any positive integers $s,k,n$ with $k>s$, so we would have that the tower height of $r(K_n^{(s)}, K_n^{(s)})$ is at least $s-1$, agreeing with the longstanding upper bound. 

To see the inequality $r(K_n^{(s)}, K_n^{(s)})\geq r(\cH_s^{(k)}, K_n^{(k)})$, suppose for the sake of contradiction that there is a red/blue-coloring of the edges of $K_N^{(k)}$ with $N = r(K_n^{(s)}, K_n^{(s)})$ that avoids any red $H$ in $\cH_s^{(k)}$.
Label the $k$ parts of each of the red $s$-tight components from $1$ to $k$.
Consider the auxiliary coloring of the edges of $K_N^{(s)}$ where an edge $e$ is red if it is contained in a red edge $\Tilde{e}$ in the original coloring and $e$ contains precisely the vertices in $\Tilde{e}$ from parts $1,2,\ldots,s$ in the red $s$-tight component of $\Tilde{e}$; $e$ is blue otherwise.
Since $\abs{e}=s$, any red $\Tilde{e}$ containing $e$ belongs to the same red $s$-tight component, showing that the color of $e$ does not depend on the choice of $\Tilde{e}$. 
It is clear that each red $\Tilde{e}$ in $K_N^{(k)}$ contains exactly one red $e$ in the auxiliary coloring of $K_N^{(s)}$.
By the definition of $r(K_n^{(s)}, K_n^{(s)})$, this auxiliary coloring contains either a red $K_n^{(s)}$ or a blue $K_n^{(s)}$.
We will show that either one would give a blue $K_n^{(k)}$ in the original coloring. 
Indeed, suppose that $S$ is the set of $n$ vertices that forms a monochromatic $K_n^{(s)}$ in the auxiliary coloring. Then each $\Tilde{e}\in \binom{S}{k}$ contains either $0$ or $\binom{k}{s}>1$ red edges $e$, so it has to be blue.
This yields the required contradiction and so $r(K_n^{(s)}, K_n^{(s)})\geq r(\cH_s^{(k)}, K_n^{(k)})$ must hold.

\paragraph{Acknowledgments.} We are grateful to Jiaxi Nie, Maya Sankar and Yuval Wigderson for stimulating conversations. We are also grateful to the anonymous reviewers for several helpful remarks.

\end{document}